\documentclass[11pt]{amsart}

\usepackage[english]{babel}
\usepackage{amsmath,amsthm,amsfonts,amssymb,stmaryrd}
\usepackage{mathtools}
\usepackage{tikz-cd}
\usepackage{tikz}
\usepackage{color}

\usepackage[colorlinks=true,linkcolor=blue]{hyperref}

\newcommand{\hooklongrightarrow}{\lhook\joinrel\longrightarrow}

\usetikzlibrary{decorations.markings}
\tikzset{double line with arrow/.style args={#1,#2}{decorate,decoration={markings,%
			mark=at position 0 with {\coordinate (ta-base-1) at (0,1pt);
				\coordinate (ta-base-2) at (0,-1pt);},
			mark=at position 1 with {\draw[#1] (ta-base-1) -- (0,1pt);
				\draw[#2] (ta-base-2) -- (0,-1pt);
}}}}
\tikzset{Equal/.style={-,double line with arrow={-,-}}}

\newcommand{\Z}{\mathbb{Z}}
\newcommand{\F}{\mathbb{F}}
\newcommand{\Fp}{\mathbb{F}_p}
\newcommand{\Zp}{\mathbb{Z}_p}
\newcommand{\p}{\mathfrak{p}}
\DeclareMathOperator{\rk}{rk}
\newcommand{\prk}{\rk_p}
\newcommand{\HH}{\mathcal{H}}
\newcommand{\Ho}{\operatorname{H}^{*}}

\DeclareMathOperator{\Ker}{Ker}
\DeclareMathOperator{\Image}{Im}
\DeclareMathOperator{\id}{id}
\DeclareMathOperator{\Hom}{Hom}
\DeclareMathOperator{\res}{res}

\DeclareMathOperator{\Ext}{Ext}
\DeclareMathOperator{\YExt}{YExt}
\DeclareMathOperator{\XExt}{XExt}
\DeclareMathOperator{\depth}{depth}

\DeclarePairedDelimiter{\abs}{\lvert}{\rvert}
\DeclarePairedDelimiter{\gen}{\langle}{\rangle}
\newcommand{\isom}{\cong}
\newcommand{\normaleq}{\unlhd}

\theoremstyle{plain}
\newtheorem{MainTheorem}{Theorem}

\newtheorem*{Notation*}{Notation}
\newtheorem{lemma}{Lemma}[section]
\newtheorem{proposition}[lemma]{Proposition}
\newtheorem{theorem}[lemma]{Theorem}

\newtheorem{conjecture}[lemma]{Conjecture}

\theoremstyle{definition}
\newtheorem{definition}[lemma]{Definition}
\newtheorem{remark}[lemma]{Remark}

\title{A family of finite $p$-groups satisfying Carlson's depth conjecture}
 
\author[O.\ Garaialde Oca\~na]{Oihana Garaialde Oca\~na}
\address{Matematika Saila,
Euskal Herriko Unibertsitatearen Zientzia eta Teknologia Fakultatea,
 posta-kutxa 644, 48080 Bilbo, Spain}
\email{oihana.garayalde@ehu.eus}

\author[L.\ Guerrero Sanchez]{Lander Guerrero Sanchez}
\address{Matematika Saila,
Euskal Herriko Unibertsitatearen Zientzia eta Teknologia Fakultatea,
 posta-kutxa 644, 48080 Bilbo, Spain}
\email{lander.guerrero@ehu.eus}

\author[J.\ Gonz\'alez-S\'anchez]{Jon Gonz\'alez-S\'anchez}
\address{Matematika Saila,
Euskal Herriko Unibertsitatearen Zientzia eta Teknologia Fakultatea,
 posta-kutxa 644, 48080 Bilbo, Spain}
\email{jon.gonzalez@ehu.eus}

\thanks{The second author was supported by the University of the Basque  Country
  predoctoral fellowship  PIF19/44 . The three authors were partially supported by the
  Spanish Government project  MTM2017-86802- P and by the Basque Government project IT974-16}

\begin{document}


\maketitle

\begin{abstract}
Let $p>3$ be a prime number and let $r$ be an integer with $1<r<p-1$. For each $r$, let moreover $G_r$ denote the unique quotient of the maximal class pro-$p$ group of size $p^{r+1}$. We show that the mod-$p$ cohomology ring of $G_r$ has depth one and that, in turn, it satisfies the equalities in Carlson's depth conjecture \cite{Carlson95}.
\end{abstract}


\section{Introduction}
\label{sec:introduction}

Let $p$ be a prime number, let $G$ be a finite $p$-group and  let $\F_p$ denote the finite field of $p$ elements with trivial $G$-action. Then, the mod-$p$ cohomology ring $\Ho(G;\F_p)$ is a finitely generated, graded commutative $\Fp$-algebra (see \cite[Corollary 7.4.6]{Evens91}), and so many ring-theoretic notions can be defined; Krull dimension, associated primes and depth, among others. Some of the aforementioned concepts have a group-theoretic interpretation; for instance, the Krull dimension $\dim \Ho(G;\Fp)$ of $\Ho(G;\F_p)$ equals the $p$-rank  $\prk G$ of $G$, i.e., the largest integer $s\geq 1$ such that $G$ contains an elementary abelian subgroup of rank $s$. However, the depth of $\Ho(G;\F_p)$, written as $\depth \Ho(G;\F_p)$, is the length of the longest regular sequence in $\Ho(G;\F_p)$, and it  seems to be far more difficult to compute.
 There are, however, lower and upper bounds for this number. For instance, in \cite{Duflot81}, Duflot proved that the depth of $\Ho(G;\F_p)$ is at least as big as the $p$-rank of the centre $Z(G)$ of $G$, i.e., $\depth \Ho(G;\Fp)\geq \prk Z(G)$ and, in \cite{Notbohm09}, Notbohm proved that for every elementary abelian subgroup $E$ of $G$ with centralizer $C_G(E)$ in $G$, the inequality $\depth \Ho(G;\Fp)\leq \depth \Ho\big(C_G(E);\Fp\big)$ holds. In \cite{Carlson95}, J. Carlson provided further upper bounds for the depth (see Theorem \ref{thm:DepthCarlson}) and stated a conjecture that still remains open (see Conjecture \ref{conj:Carlson}).

The aim of the present work is to compute the depth of the mod-$p$ cohomology rings of certain quotients of the maximal class pro-$p$ group
that moreover satisfy the equalities in the aforementioned conjecture.
Let $p$ be an odd prime number, let $\Z_p$ denote the ring of $p$-adic integers and let $\zeta$ be a primitive $p$-th root of unity. Consider the  cyclotomic extension $\Zp[\zeta]$ of degree $p-1$ and note that its additive group is isomorphic to $\Zp^{p-1}$. The cyclic group $C_p=\gen{\sigma}$ acts on $\Zp[\zeta]$ via multiplication by $\zeta$, i.e., for any $x\in\Z_p$, the action is given as $x^{\sigma}=\zeta x$. Using the ordered basis $1, \zeta, \dots, \zeta^{p-2}$ in $\Zp[\zeta]\isom \Zp^{p-1}$, this action is given by the matrix
\begin{displaymath}
	\begin{pmatrix}
		0		& 1 	 & 0 	  & \dots  & 0 		\\
		0		& 0 	 & 1 	  & \dots  & 0 		\\
		\vdots  & \vdots & \vdots & \ddots & \vdots \\
		0		& 0		 & 0	  & \dots  & 1		\\
		-1		& -1	 & -1	  & \dots  & -1		
	\end{pmatrix}.
\end{displaymath}
We form the semidirect product $S=C_p\ltimes \Zp^{p-1}$, which is the unique pro-$p$ group of maximal nilpotency class. Note that this is the analogue of the infinite dihedral pro-$2$ group for the $p$ odd case. Moreover, $S$ is a uniserial $p$-adic space group with cyclic point group $C_p$ (compare \cite[Section 7.4]{LeedGreenBook02}). We write $[x, _{k}\sigma]=[x, \sigma, \overset{k}{\vphantom{,}\smash{\dotsc}}\,, \sigma]$ for the iterated group commutator. 
Set $T_0=\Zp[\zeta]$ and define, for each integer $i\geq 1$, 
\[
T_i=(\zeta-1)^i\Zp[\zeta]=[T_0,_i \sigma]=\gamma_{i+1}(S).
\]
These subgroups are all the $C_p$-invariant subgroups of $T_0$, and the successive quotients satisfy
\begin{displaymath}
	T_i/T_{i+1}\isom \Zp[\zeta]/(\zeta-1)\Zp[\zeta]\isom C_p.
\end{displaymath}
Hence, $\abs{T_0:T_i}=p^i$ for every $i\geq 0$. For each integer $r$ with $1<r<p-1$, consider the finite quotient $T_0/T_r=\Zp[\zeta]/(\zeta-1)^r\Zp[\zeta]$ and choose a generating set for $T_0/T_r$ as follows,
\begin{displaymath}
	a_1=1+T_r,\quad a_2=(\zeta-1)+T_r,\quad\dots ,\quad a_r=(\zeta-1)^{r-1}+T_r.
\end{displaymath}
Using the multiplicative notation, we obtain that
$$
T_0/T_r=\gen{a_1,\dotsc,a_r}\isom C_p\times \overset{r}{\vphantom{=}\smash{\dotsb}} \times C_p,
$$
and since for all $i\geq 0$, the subgroups $T_i$ are $C_p$-invariant, we can form the semidirect product
 \begin{equation}\label{eq:Grassemidirectproduct}
 G_r=C_p\ltimes T_0/T_r\isom C_p\ltimes(C_p\times \overset{r}{\vphantom{=}\smash{\dotsb}} \times C_p).
 \end{equation}
The finite $p$-groups $G_r$ have size $p^{r+1}$ and exponent $p$. Note that in particular, $G_2$ is the extraspecial group of size $p^3$ and exponent $p$. We state the main result.

\begin{MainTheorem}\label{thm:mainresultintro}
Let $p>3$ be a prime number, let $r$ be an integer with $1<r<p-1$ and let $G_r$ be given as in \eqref{eq:Grassemidirectproduct}. Then, $\Ho(G_r;\F_p)$ has depth one.
\end{MainTheorem}

\subsubsection*{Notation}\label{notation}
Throughout let $p$ be an odd prime number and let $G$ denote a finite group. Let $R$ be a commutative ring with unity. A $G$-module $A$ will be a right $RG$-module. For such $G$-modules, we shall use additive notation in Sections \ref{sec:preliminaries} and \ref{sec:productsextensions}, and multiplicative notation in Section \ref{sec:depthonepgroups}, for our convenience. Moreover, if $a\in A$ and $g\in G$, we write $a^g$ to denote the action of $g$ on $a$.

Let $A$ be a $G$-module and let $P_{*}\longrightarrow R$ be a projective resolution of the trivial $G$-module $R$, then for every $n\geq 0$, the $n$-th cohomology group $\operatorname{H}^n(G;A)$ is defined as $\Ext^n(R,A)=\operatorname{H}^n(\Hom_G(P_{*},A))$. Let $K\leq G$ be a subgroup of $G$ and let $\iota\colon K\hooklongrightarrow G$ denote an inclusion map. This map induces the restriction map in cohomology, which will be denoted by $\res^G_K\colon \Ho(G;A)\longrightarrow \Ho(K;A)$.

Group commutators are given as $[g,h]=g^{-1}h^{-1}gh=g^{-1}g^h$ and for every $k\geq 1$, iterated commutators are written as $[x, y, \overset{k}{\vphantom{,}\smash{\dotsc}}\,, y]=[x, _{k}y]$, where we use left normed group commutators, i.e., $[x,y,z]=[[x,y],z]$. Also, the $k$-th term of the lower central series of $G$ is denoted by $\gamma_k(G)=[G, \overset{k}{\vphantom{,}\smash{\dotsc}}\,, G]$.

\section{Preliminaries}\label{sec:preliminaries}

\subsection{Depth}

In this section we give background on the depth of mod-$p$ cohomology rings of finite $p$-groups and we also state one of the key results for the proof of Theorem \ref{thm:mainresultintro}.

Let $n\geq 1$ be an integer number. We say that a sequence of elements $x_1,\dotsc,x_n\in \Ho(G;\Fp)$ is regular if, for every $i=1, \dots, n$, the element $x_i$ is not a zero divisor in the quotient $\Ho(G;\Fp)/(x_1,\dotsc,x_{i-1})$, where $(x_1,\dotsc,x_{i-1})$ denotes the ideal generated by the elements $x_1, \dots, x_{i-1}$ in $\Ho(G;\F_p)$.

\begin{definition} 
The \emph{depth} of $\Ho(G;\Fp)$, denoted by $\depth \Ho(G;\Fp)$, is the maximal length of a regular sequence in $\Ho(G;\Fp)$. 
\end{definition}

Recall that a prime ideal $\p\subseteq \Ho(G;\Fp)$ is an \emph{associated prime} of $\Ho(G;\Fp)$ if, for some $\varphi\in \Ho(G;\Fp)$, it is of the form 
\begin{displaymath}
\p=\{\psi\in \Ho(G;\Fp)\mid \varphi\cup\psi=0\}.
\end{displaymath}
The set of all associated primes of $\Ho(G;\Fp)$ is denoted by $\operatorname{Ass} \Ho(G;\Fp)$. It is known that for every $\p\in \operatorname{Ass}\Ho(G;\Fp)$, the following inequality holds 
	\begin{displaymath}
	\depth \Ho(G;\Fp)\leq \dim \Ho(G;\Fp)/\p.
	\end{displaymath}
In particular, $\depth \Ho(G;\Fp)\leq \dim \Ho(G;\Fp)$ (\cite[Proposition 12.2.5]{CarlsonBook03}) and, when the two values coincide, the mod-$p$ cohomology ring is said to be \emph{Cohen-Macaulay}. In the following proposition, we recall the lower and upper bounds for the depth of $\Ho(G;\F_p)$ by Duflot \cite{Duflot81} and Notbohm \cite{Notbohm09}, respectively.


\begin{proposition}\label{prop:DuflotNotbohm} Let $G$ be a finite $p$-group. The following inequalities hold
$$
1\leq \prk Z(G)\leq \depth \Ho(G;\Fp)\leq \depth \Ho\big(C_G(E);\Fp\big).
$$ 
\end{proposition}

Before stating the crucial result for our construction, we introduce the concept of detection in cohomology.

\begin{definition}\label{def:detection}
	Let $G$ be a finite $p$-group and let $\HH$ be a collection of subgroups of $G$. We say that $\Ho(G;\Fp)$ is \emph{detected} by $\HH$ 
	if 
	\begin{displaymath}
	\bigcap_{H\in \HH}\Ker \res^G_H=0.
	\end{displaymath}
\end{definition}

Given a finite $p$-group $G$ and a subgroup $E\leq G$, let $C_G(E)$ denote the centralizer of $E$ in $G$. For $s\geq 1$, define
\begin{align*}
&\HH_s(G)=\big\{C_G(E) \mid E \text{ is an elementary abelian subgroup of } G,\; \prk E=s\big\},\\
&\omega_a(G)=\min \big\{\dim \Ho(G;\Fp)/\p\mid \p\in\operatorname{Ass}\Ho(G;\Fp)\big\},\\
&\omega_d(G)=\max\big\{s\geq 1\mid \Ho(G;\Fp) \text{ is detected by } \HH_s(G)\big\}.
\end{align*}

\begin{theorem}[\cite{Carlson95}]\label{thm:DepthCarlson}
	Let $G$ be a finite $p$-group. Then, the following inequalities hold
	\begin{displaymath}
	\depth\Ho(G;\Fp)\leq\omega_a(G)\leq \omega_d(G).
	\end{displaymath}
\end{theorem}

In fact, in the same article, J. F. Carlson conjectured that the previous inequalities are actual equalities.

\begin{conjecture}[Carlson]\label{conj:Carlson}
	Let $G$ be a finite $p$-group. Then,
	\begin{displaymath}
	\depth\Ho(G;\Fp)=\omega_a(G)= \omega_d(G).
	\end{displaymath}
\end{conjecture}
A particular case of the above conjecture was proven by D. Green in \cite{Green03}.

\subsection{Yoneda extensions}

Let $G$ be a finite group and let $R$ be a commutative ring with unity. We describe the mod-$p$ cohomology ring $\Ho(G;R)$ in terms of Yoneda extensions and the Yoneda product. For a more detailed account on this topic, we refer to \cite[Chapter IV]{Maclane95} and \cite{Niwasaki92}.

\begin{definition}

Let $A$ and $B$ be $G$-modules. For every integer $n\geq 1$,  a \emph{Yoneda $n$-fold extension} $\varphi$ \emph{of $B$ by $A$} is an exact sequence of $G$-modules of the form
\[\varphi:\begin{tikzcd}
0 \rar & A \rar & M_n \rar & \cdots \rar & M_1 \rar & B \rar & 0.
\end{tikzcd}\]
Given two Yoneda $n$-fold extensions
\[\varphi:\begin{tikzcd}
0 \rar & A \rar & M_n \rar & \cdots \rar & M_1 \rar & B \rar & 0
\end{tikzcd}\]
and 
\[\varphi':\begin{tikzcd}
0 \rar & A' \rar & M_n' \rar & \cdots \rar & M_1' \rar & B' \rar & 0,
\end{tikzcd}\]
we say that there is a \emph{morphism between the Yoneda $n$-fold extensions $\varphi$ and $\varphi'$}, if there exist  $G$-module homomorphisms $f_0\colon B\longrightarrow B'$, $f_{n+1}\colon A\longrightarrow A'$ and, for every $i=1,\dots,n$, $f_i\colon M_i\longrightarrow M'_i$,  making the following diagram commute
\[
\begin{tikzcd}
\varphi: 0 \rar & A \dar{f_{n+1}} \rar & M_n \dar{f_n} \rar & \cdots \rar & M_1 \dar{f_1} \rar & B \dar{f_0} \rar & 0\\
\varphi': 0 \rar & A' \rar & M_n' \rar & \cdots \rar & M_1' \rar & B' \rar & 0.
\end{tikzcd}\]

In particular, if $\varphi, \varphi'$ are both Yoneda $n$-fold extensions of $A$ by $B$ and, there is a morphism from $\varphi$ to $\varphi'$ with identity maps $f_0=\id_B$ and $f_{n+1}=\id_A$, we write $\varphi\Rightarrow \varphi'$.
\end{definition}


\begin{definition}\label{def:Yequivalent}
Let $n\geq 1$ be an integer and let $\varphi$ and $\varphi'$ be as above. We say that $\varphi$ is \emph{equivalent} to $\varphi'$, denoted by $\varphi\equiv \varphi'$, if there are Yoneda $n$-fold extensions $\varphi_1,\dotsc,\varphi_r$ of $B$ by $A$ such that 
\begin{displaymath}
\varphi \Rightarrow \varphi_1  \Leftarrow \varphi_2 \Rightarrow \cdots \Leftarrow \varphi_{r-1} \Rightarrow \varphi_r \Leftarrow \varphi'.
\end{displaymath}
Moreover, we denote by $\YExt^n(B,A)$ the set of all Yoneda $n$-fold extensions of $B$ by $A$ up to equivalence.
\end{definition}

We recall the uniqueness of pushouts and pullbacks of Yoneda extensions \cite[Section II.6]{Hilton97} and endow the set $\YExt^n(B,A)$ with the Baer sum so that $\YExt^n(B,A)$ becomes an abelian group. The proof of the following result can be found in  \cite[Section IV.9]{Hilton97}.

\begin{proposition}
Let $\varphi\in \YExt^n(B,A)$ be represented by a Yoneda extension
\[\begin{tikzcd}
0 \rar & A \rar & M_n \rar & \cdots \rar & M_1 \rar & B \rar & 0.
\end{tikzcd}\]
\begin{enumerate}
\item [(a)]Given a $G$-module homomorphism $\alpha\colon A\longrightarrow A'$, there is a unique equivalence class $\alpha_*\varphi\in\YExt^n(B,A')$ represented by a Yoneda extension
\[\begin{tikzcd}
0 \rar & A' \rar & M_n' \rar & \cdots \rar & M_1' \rar & B \rar & 0,
\end{tikzcd}\]
admitting a morphism of Yoneda extensions of the following form:
\[\begin{tikzcd}
0 \rar & A \rar \dar{\alpha} & M_n \dar \rar & \dotsb \rar & M_1 \rar \dar & B \rar \dar[Equal] & 0 \\
0 \rar & A' \rar & M_n' \rar & \dotsb \rar & M_1' \rar & B \rar & 0.
\end{tikzcd}\]
We say that the Yoneda extension $\alpha_*\varphi$ is the \emph{pushout of $\varphi$ via $\alpha$}.
\item [(b)]Given a $G$-module homomorphism $\beta\colon B'\longrightarrow B$ there is a unique equivalence class $\beta^*\varphi\in\YExt^n(B',A)$ represented by a Yoneda extension 
\[\begin{tikzcd}
0 \rar & A \rar & M_n'' \rar & \cdots \rar & M_1'' \rar & B' \rar & 0,
\end{tikzcd}\]
admitting a morphism of Yoneda extensions of the following form:
\[\begin{tikzcd}
0 \rar & A \rar \dar[Equal] & M_n'' \dar \rar & \dotsb \rar & M_1'' \rar \dar & B' \rar \dar{\beta} & 0 \\
0 \rar & A \rar & M_n \rar & \dotsb \rar & M_1 \rar & B \rar & 0.
\end{tikzcd}\]
We say that the Yoneda extension $\beta^*\varphi$ is the \emph{pullback} of $\varphi$ via $\beta$.
\end{enumerate}
\end{proposition}

Let now $A$ and $B$ be $G$-modules and let 
\[
\nabla_A\colon A\times A\longrightarrow A\;\; \text{and}\;\; \Delta_B\colon B\longrightarrow B\times B
\]
denote the codiagonal and the diagonal homomorphism, respectively.

\begin{definition}
Let $n\geq 1$ be an integer and let $\varphi,\varphi'\in\YExt^n(B,A)$ be two Yoneda extension classes.  We define the \emph{Baer sum} of $\varphi$ and $\varphi'$  as 
\begin{displaymath}
\varphi +\varphi'=(\nabla_A)_*(\Delta_B)^*(\varphi\times\varphi')\in \YExt^n(B,A).
\end{displaymath}
\end{definition}
Then, for every integer $n\geq1$, the set $\YExt^n(B,A)$ endowed with the Baer sum is an abelian group. Indeed, the zero element of $\YExt^1(B,A)$ is the split extension
\[\begin{tikzcd}
0 \rar & A \rar & A\times B \rar & B \rar & 0,
\end{tikzcd}\]
and for $n>1$, the zero element of $\YExt^n(B,A)$ is the Yoneda extension
\[\begin{tikzcd}[column sep=2em]
0 \rar & A \rar & A \rar & 0 \rar & \overset{n-2}{\vphantom{=}\smash{\cdots}} \rar & 0\rar & B \rar & B \rar & 0.
\end{tikzcd}\]

\begin{theorem}[{\cite[Theorem 6.4]{Maclane95}}]
	For every $G$-module $A$ and integer $n\geq 1$, there is a group isomorphism $\operatorname{H}^n(G;A)\isom \YExt^{n}(R,A)$ that is natural in  $A$.
\end{theorem}

Let $A$, $B$ and $C$ be $G$-modules and let  $n,m\geq 1$ be integers. Given $\varphi\in\YExt^n(B,A)$ represented by the Yoneda extension
\[\begin{tikzcd}
0 \rar & A \rar & N_n \rar & \cdots \rar & N_1 \rar & B \rar & 0,
\end{tikzcd}\]
and $\varphi'\in\YExt^m(C,B)$ represented by the Yoneda extension
\[\begin{tikzcd}
0 \rar & B \rar & M_m \rar & \cdots \rar & M_1 \rar & C \rar & 0,
\end{tikzcd}\]
we define their \emph{Yoneda product} $\varphi\cup\varphi'\in \YExt^{n+m}(G,B)
$ as the Yoneda extension
\[\begin{tikzcd}[column sep=1.5em]
0 \rar & A \rar & N_n \rar & \cdots \rar & N_1 \rar & M_m \rar & \cdots \rar & M_1 \rar & C \rar & 0.
\end{tikzcd}\]
This product defines a bilinear pairing. In particular, the Yoneda product in $\Ho(G;R)\isom\YExt^{*}(R,R)$ coincides with the usual cup product (see \cite[Proposition 3.2.1]{Benson91}).

\subsection{Crossed extensions}

Let $G$ be a finite  group. In this section we describe, for every integer $n\geq 2$,  the cohomology group $\operatorname{H}^n(G; R)$ using crossed extensions. For more information about this subject, see \cite{Holt79}, \cite{Huebschmann80} and \cite{Niwasaki92}.

\begin{definition}
	Let $M_1$ and $M_2$ be groups with $M_1$ acting on $M_2$. A \emph{crossed module} is a group homomorphism $\rho\colon M_2\longrightarrow M_1$ satisfying the following properties:
\begin{enumerate}
	\item[(i)] $y_2^{\rho(y'_2)}=y_2^{y'_2}$ for all $y_2,y'_2\in M_2$, and
	\item[(ii)] $\rho(y_2^{y_1})=\rho(y_2)^{y_1}$ for all $y_1\in M_1$ and $y_2\in M_2$.
\end{enumerate}
\end{definition}

\begin{definition}
Let $n\geq 1$ be an integer and let $A$ be a $G$-module. A \emph{crossed $n$-fold extension $\psi$ of $G$ by $A$} is an exact sequence of groups of the form  
\[\begin{tikzcd}
\psi: 0 \rar & A \rar{\rho_{n}} & M_n \rar & \cdots \rar & M_2 \rar{\rho_1} & M_1 \rar & G \rar & 1,
\end{tikzcd}\]
satisfying the following conditions:
\begin{enumerate}
	\item[(i)] $\rho_1\colon M_2\longrightarrow M_1$ is a crossed module,
	\item[(ii)] $M_i$ is a $G$-module for every $i=3,\dotsc,n$, and
	\item[(iii)] $\rho_i$ is a $G$-module homomorphism for every $i=2,\dotsc,n$.
\end{enumerate} 
\end{definition}

\begin{definition}
A \emph{morphism of crossed $n$-fold extensions} $\psi$ and $\psi'$ is a morphism of exact sequences of groups 
\[\begin{tikzcd}[column sep=2em]
\psi: 0 \rar & A \dar{f_{n+1}} \rar & M_n \dar{f_n} \rar & \cdots \rar & M_2 \dar{f_2} \rar & M_1 \dar{f_1} \rar & G \dar{f_0} \rar & 1 \\
\psi': 0 \rar & A' \rar & M_n' \rar & \cdots \rar & M_2' \rar & M_1' \rar & G' \rar & 1,
\end{tikzcd}\]
where for each $i=3,\dotsc,n+1$, the morphism $f_i$ is a $G$-module homomorphism and, $f_1$ and $f_2$ are compatible with the actions of $M_1$ on $M_2$ and of $M_1'$ on $M_2'$, respectively. 

In particular, if $\psi$ and $\psi'$ are both crossed $n$-fold extensions of $G$ by $A$ and, there is a morphism from $\psi$ to $\psi'$ with identity maps $f_0=\id_G$ and $f_{n+1}=\id_A$, we write $\psi \Rightarrow \psi'$.
\end{definition}

 Moreover, we can define an equivalence relation on crossed $n$-fold extensions of $G$ by $A$ as for Yoneda extensions in Definition \ref{def:Yequivalent}, denoted by $\psi\equiv\psi'$. We will also denote by $\XExt^n(G,A)$ the set of all crossed $n$-fold extensions of $G$ by $A$ up to equivalence.

For the $n=2$ case, we can use the following characterization of equivalent crossed extensions.

\begin{proposition}[\cite{Holt79}]\label{prop: Equivalent Crossed X Diagram}
	Let $G$ be a finite group and let $A$ be a $G$-module. Then, two crossed $2$-fold extensions of $G$ by $A$
	\begin{footnotesize}
	\begin{equation*}
	\psi: 0 \longrightarrow  A \overset{\rho_2}\longrightarrow  M_2 \overset{\rho_1}\longrightarrow M_1 \overset{\rho_0}\longrightarrow  G \longrightarrow  1\;  \text{ and } \; \psi': 0 \longrightarrow  A \overset{\tau_2}\longrightarrow  N_2 \overset{\tau_1}\longrightarrow  N_1 \overset{\tau_0}\longrightarrow  G \longrightarrow  1,
	\end{equation*}
	\end{footnotesize}are equivalent if and only if there exist a group $X$ and a commutative diagram
	\begin{equation}\label{diag: Crossed diagram}
	\begin{tikzcd}[column sep={4.em,between origins},row sep=2em]
	& 1 \drar & & & & 1 & \\
	& & M_2 \drar[swap]{\mu_1} \arrow[rr,"\rho_1"] & & M_1 \drar{\rho_0} \urar & & \\
	0 \rar & A \urar{\rho_2} \drar[swap]{-\tau_2} & & X \urar[swap]{\nu_1} \drar{\nu_2} & & G \rar & 1 \\
	& & N_2 \urar{\mu_2} \arrow[rr,swap,"\tau_1"] & & N_1 \arrow[ur,swap,"\tau_0"] \drar & & \\
	& 1 \urar &  & & & 1 & 
	\end{tikzcd}
	\end{equation}
	satisfying the following properties: 
	\begin{itemize}
	\item[(a)] $-\tau_2\colon A\longrightarrow N_2$ is given by $(-\tau_2)(a)=\tau_2(-a)$ for $a\in A$, 
	\item[(b)] the diagonals are short exact sequences, 
	\item[(c)] $\mu_1\circ \rho_2(A)=\mu_1(M_2)\cap\mu_2(N_2)$, and 
	\item[(d)] conjugation in $X$ coincides with the actions of both $M_1$ on $M_2$ and $N_1$ on $N_2$.
	\end{itemize}
\end{proposition}

Analogous to Yoneda extensions, for an integer $n\geq 1$, given an $n$-crossed extension $\varphi\in \XExt^n(G,A)$ and  a $G$-module homomorphism $\alpha\colon A\longrightarrow A'$, we can find a unique pushout $\alpha_*\varphi\in \XExt^n(G,A')$ of $\varphi$ via $\alpha$, and given a group homomorphism $\beta\colon G'\longrightarrow G$ we can find a unique pullback $\beta^*\varphi\in\XExt^n(G',A)$ of $\varphi$ via $\beta$ (see \cite[Proposition 4.1]{Holt79}).

We can also endow $\XExt^n(G,A)$ with an abelian group structure. Given two crossed $n$-fold extension classes $\varphi,\varphi'\in\XExt^n(G,A)$, we define their \emph{Baer sum} as 
\begin{displaymath}
\varphi +\varphi'=(\nabla_A)_*(\Delta_G)^*(\varphi\times\varphi').
\end{displaymath}

The zero element of $\XExt^1(G,A)$ is represented by the split extension
\[\begin{tikzcd}
0 \rar & A \rar & G\ltimes A \rar & G \rar & 1,
\end{tikzcd}\]
and for $n>1$, the zero element of $\XExt^n(G,A)$ is represented by the Yoneda extension
\[\begin{tikzcd}[column sep=2em]
0 \rar & A \rar & A \rar & 0 \rar & \overset{n-2}{\vphantom{=}\smash{\cdots}} \rar & 0\rar & G \rar & G \rar & 1.
\end{tikzcd}\]

\begin{theorem}[{\cite[Theorem 4.5]{Holt79}}]
	Let $G$ be a finite group. For every $G$-module $A$ and every integer $n\geq 1$, there is a group isomorphism $\operatorname{H}^{n+1}(G;A)\isom \XExt^{n}(G,A)$ that is natural in both $G$ and $A$.
\end{theorem}

\section{Product between extensions}\label{sec:productsextensions}

\subsection{Product of Yoneda extensions and crossed extensions}

We now describe the Yoneda product between two cohomology classes, one of them represented by a Yoneda extension and the other one by a crossed extension.

\begin{definition}\label{def:mixedyonedaproduct}
Let $G$ be a finite group, let $A$ and $B$ be $G$-modules and let $n,m\geq 1$ be integer numbers. Given a Yoneda $n$-fold extension class $\varphi\in \YExt^n(A,B)$ represented by
\[\begin{tikzcd}
0 \rar & B \rar & N_n \rar
& \cdots \rar & N_1 \rar & A \rar & 0,
\end{tikzcd}\]
and a crossed $m$-fold extension class $\psi\in\XExt^m(G,A)$ represented by
\[\begin{tikzcd}
0 \rar & A \rar & M_m \rar & \cdots \rar & M_1 \rar & G \rar & 1,
\end{tikzcd}\]
we define their \emph{Yoneda product} $\varphi\cup\psi$ as the extension
\[\begin{tikzcd}[column sep=1.5em]
0 \rar & B \rar & N_n \rar
& \cdots \rar & N_1 \rar & M_m \rar & \cdots \rar & M_1 \rar & G \rar & 1.
\end{tikzcd}\]
\end{definition}

\begin{remark}
It can be readily checked that 
\[
\YExt^n(A,B)\times \XExt^m(G,A)\longrightarrow \XExt^{n+m}(G,B)
\]
given by $(\varphi, \psi)\mapsto \varphi \cup \psi$ is well defined.
\end{remark}

The following result shows that this product respects the pushouts and pullbacks.

\begin{lemma}
	Let $G$ and $G'$ be finite groups, let $A$, $A'$, $B$ and $B'$ be $G$-modules and let $n,m\geq 1$ be integer numbers. Let, moreover, $\varphi\in\YExt^n(A,B)$, $\varphi'\in\YExt^n(A',B)$ and $\psi\in\XExt^m(G,A)$. Then, the following relations are satisfied.
	\begin{enumerate}
		\item Given a $G$-module homomorphism $\alpha\colon A\longrightarrow A'$, we have that 
		\begin{displaymath}
			(\alpha^*\varphi')\cup\psi \equiv \varphi'\cup (\alpha_*\psi)\in \XExt^{n+m}(G,B).
		\end{displaymath}
		
		\item Given a $G$-module homomorphism $\beta\colon B\longrightarrow B'$, we have that
		\begin{displaymath}
			(\beta_*\varphi)\cup\psi\equiv\beta_*(\varphi\cup\psi)\in \XExt^{n+m}(G,B').
		\end{displaymath}
		
		\item Given a group homomorphism $\tau\colon G'\longrightarrow G$, we have that
		\begin{displaymath}
			\varphi\cup(\tau^*\psi)\equiv\tau^*(\varphi\cup\psi)\in \XExt^{n+m}(G',B).
		\end{displaymath}
	\end{enumerate}
\end{lemma}
\begin{proof}
	The proofs of 2 and 3 are straightforward. For 1, follow the proof of the analogous result for Yoneda extensions mutatis mutandis (compare \cite[Proposition III.5.2]{Maclane95}).
\end{proof}

\begin{proposition}
	Let $G$ be a finite group, let $A$ and $B$ be $G$-modules and let $n,m\geq 1$ be integers. Then, the Yoneda product induces a well-defined bilinear pairing
	\[\begin{tikzcd}
	\YExt^n(A,B) \otimes \XExt^m(G,A) \rar & \XExt^{n+m}(G,B).
	\end{tikzcd}\]
\end{proposition}

\begin{proof}
	Let $\varphi,\varphi'\in \YExt^n(A,B)$ and $\psi,\psi'\in\XExt^m(G,A)$. On the one hand, using that $(\Delta_A)_*\psi\equiv(\Delta_G)^*(\psi\times\psi)$, we have that 
	\begin{align*}
		(\varphi + \varphi')\cup \psi & \equiv \big[(\nabla_B)_*(\Delta_A)^*(\varphi\times\varphi')\big]\cup \psi \\
		& \equiv (\nabla_B)_*(\Delta_G)^* \big[(\varphi\times\varphi')\cup(\psi\times\psi)\big] \\
		& \equiv (\nabla_B)_*(\Delta_G)^* \big[(\varphi\cup\psi)\times(\varphi'\cup\psi)\big] \\
		& \equiv \varphi\cup \psi + \varphi'\cup \psi.
	\end{align*}
	On the other hand, we have that 
	\begin{align*}
		\varphi\cup (\psi+\psi') & \equiv \varphi \cup \big[(\nabla_A)_*(\Delta_G)^*(\psi\times\psi')\big] \\
		& \equiv (\nabla_B)_*(\Delta_G)^* \big[(\varphi\times\varphi)\cup(\psi\times\psi')\big] \\
		& \equiv (\nabla_B)_*(\Delta_G)^* \big[(\varphi\cup\psi)\times(\varphi\cup\psi')\big] \\
		& \equiv \varphi\cup \psi + \varphi\cup \psi'.
	\end{align*}
\end{proof}

\subsection{Yoneda and cup products coincide}

In order to show that the Yoneda product of Yoneda extensions with crossed extensions coincides with the usual cup product, we will follow a construction by B. Conrad \cite{conrad}, giving an explicit correspondence between crossed extensions and Yoneda extensions.

Let $G$ be a finite group and let $A$ be a $G$-module. Let $\psi\in\XExt^n(G,A)$ be a class represented by a crossed $n$-fold extension 
\[\begin{tikzcd}
0 \rar & A \rar{\rho_n} & M_n \rar & \cdots \rar & M_2 \rar{\rho_1} & M_1 \rar{\rho_0} & G \rar & 1,
\end{tikzcd}\]
with $M_2$ abelian (such a representative always exists, see \cite[Proposition 2.7]{Holt79}). Consider the $G$-module $\Image\rho_1\leq M_1$. Then, we have an extension $\psi_0\in\XExt^1(G,\Image\rho_1)$ of the form
\[\begin{tikzcd}
\psi_0: \;\; 0 \rar & \Image \rho_1 \rar & M_1 \rar & G \rar & 1. 
\end{tikzcd} \label{ext: conrad}\]
Now, we can embed $\Image \rho_1$ into an injective $G$-module $I$. As $I$ is injective, we have that $\XExt^1(G,I)\isom \operatorname{H}^2(G,I)=0$, and so the pushout of $\psi_0$ 
via the embedding of $\Image\rho_1$ into $I$ splits, i.e., there is a group homomorphism $\Phi\colon M_1\longrightarrow G\ltimes I$ such that the following diagram commutes:  
\[\begin{tikzcd}
0 \rar & \Image \rho_1 \rar \dar[hook] & M_1 \rar{\rho_0} \dar{\Phi} & G \rar \dar[Equal] & 1 \\
0 \rar & I \rar & G\ltimes I \rar & G \rar & 1.
\end{tikzcd}\]
We can find a group homomorphism $\nu\colon M_1\longrightarrow G$ and a map $\chi\colon M_1\longrightarrow I$ that for every $x,y\in M_1$ satisfies 
\begin{align} \label{eqn: cocycle cond semidir}
\chi(xy)=\chi(x)^{\nu(y)}\chi(y),
\end{align}
such that for every $x,y\in M_1$ we can write
\begin{displaymath}
\Phi(x)=\big(\nu(x),\chi(x)\big).
\end{displaymath}

Moreover, if we denote by $\pi\colon I\longrightarrow I/\Image\rho_1$ the canonical projection,  there is a unique map $\tau\colon G\longrightarrow I/\Image\rho_1$ such that $\tau\circ \nu=\pi\circ \chi$. Furthermore, because $\chi$ satisfies (\ref{eqn: cocycle cond semidir}) and $\nu=\rho_0$ is surjective, we have that for every $g,h\in G$,
\begin{displaymath}
	\tau(gh)=\tau(g)^h+\tau(h),
\end{displaymath}
and so $\tau$ is a $1$-cocycle. 
Hence, $\tau$ can be represented as a cohomology class in $\operatorname{H}^1(G,I/\Image\rho_1)\isom \YExt^1(R,I/\Image\rho_1)$ by a Yoneda extension of the form 
\[\begin{tikzcd}
0 \rar & I/\Image\rho_1 \rar & E_{\tau} \rar & R \rar & 0.
\end{tikzcd}\]

\begin{remark}\label{rmk: conrad choice r1}
	The choices of the $G$-module $I$ and the cocycle $\tau$, and consequently $E_{\tau}$, only depend on $\Image\rho_1\leq M_1$.
\end{remark}

Finally, we can construct the element $\Upsilon(\psi)\in\YExt^{n+1}(R,A)$ given by the Yoneda extension
\begin{equation}\label{eq:Upsilon}
0 \longrightarrow  A \longrightarrow  M_n \longrightarrow  \cdots \longrightarrow  M_2 \longrightarrow  I \longrightarrow  E_{\tau} \longrightarrow  R \longrightarrow  0.
\end{equation}
This construction gives rise to a group isomorphism $$\Upsilon\colon \XExt^n(G,A)\longrightarrow \YExt^{n+1}(R,A).$$

\begin{proposition}
	Let $G$ be a finite group and let $n,m\geq 1$ be integer numbers. Then, the Yoneda product 
	\[\begin{tikzcd}
	\YExt^n(A,B) \otimes \XExt^m(G,A) \rar & \XExt^{n+m}(G,B)
	\end{tikzcd}\]
	coincides with the Yoneda product
	\[\begin{tikzcd}
	\YExt^n(A,B) \otimes \YExt^{m+1}(R,A) \rar & \XExt^{n+m+1}(R,B).
	\end{tikzcd}\]
	In particular, if $A=B=R$, the above product coincides with the cup product
	\[\begin{tikzcd}
	\cup: \;\operatorname{H}^n(G;R) \otimes \operatorname{H}^{m+1}(G;R) \rar & \operatorname{H}^{n+m+1}(G;R).
	\end{tikzcd}\]
\end{proposition}

\begin{proof}
	Let $\varphi\in \YExt^n(A,B)$ be a class represented by an extension
	\[\begin{tikzcd}
	0 \rar & B \rar & N_n \rar
	& \cdots \rar & N_1 \rar{\mu_0} & A \rar & 0,
	\end{tikzcd}\]
	and let $\psi\in\XExt^m(G,A)$ be a class represented by an extension 
	\[\begin{tikzcd}
	0 \rar & A \rar{\rho_m} & M_m \rar & \cdots \rar & M_2 \rar{\rho_1} & M_1 \rar & G \rar & 1,
	\end{tikzcd}\]
	with $M_2$ abelian. We need to prove that $\Upsilon(\varphi\cup\psi)=\varphi\cup\Upsilon(\psi)$. 
	
	By \eqref {eq:Upsilon}, for $m>1$, the extension $\Upsilon(\psi)\in\YExt^{m+1}(R,A)$ is of the form
	\[\begin{tikzcd}[column sep=2em]
	0 \rar & A \rar & M_m \rar & \cdots \rar & M_2 \rar & I \rar & E_{\tau} \rar & R \rar & 0,
	\end{tikzcd}\]
	and $\varphi\cup\psi\in\XExt^{n+m}(G,A)$ is represented by the crossed $(n+m)$-fold extension
	\[\begin{tikzcd}[column sep=1.5em]
	0 \rar & B \rar & N_n \rar & \cdots \rar & N_1 \rar & M_m \rar & \cdots \rar & M_1 \rar & G \rar & 1.
	\end{tikzcd}\] 
	By Remark~\ref{rmk: conrad choice r1}, we can use the same $I$ and $\tau$ in the construction of $\Upsilon(\psi)$. 
	Therefore, $\Upsilon(\varphi\cup\psi)\in\YExt^{n+m+1}(G,A)$ is represented by 
	\[\begin{tikzcd}[column sep=0.95em]
	0 \rar & B \rar & N_n \rar & \cdots \rar & N_1 \rar & M_m \rar & \cdots \rar & M_2 \rar & I \rar & E_{\tau} \rar & R\rar & 0,
	\end{tikzcd}\] 
	which coincides with $\varphi\cup\Upsilon(\psi)$.

	For $m=1$, we have that $\psi\in\XExt^1(G,A)$ is represented by a crossed 1-fold extension of the form
	\[\begin{tikzcd}
	0 \rar & A \rar{\rho_1} & M_1 \rar & G \rar & 1.
	\end{tikzcd}\]
	Then, $\varphi\cup\psi$ is given by the crossed $(n+1)$-fold extension 
	\[\begin{tikzcd}
	0 \rar & B \rar & N_n \rar
	& \cdots \rar & N_1 \rar{\gamma_1} & M_1 \rar & G\rar & 1,
	\end{tikzcd}\] 
	where $\gamma_1=\rho_1\circ \mu_0$. Now, we have that $\Image \gamma_1= \Image \rho_1$, and so we can once again use the same $I$ and $\tau$ in the construction of both $\Upsilon(\psi)$ and  $\Upsilon(\varphi\cup\psi)$. Therefore, both $\varphi\cup\Upsilon(\psi)$ and  $\Upsilon(\varphi\cup\psi)$ are given by the same extension
	\[\begin{tikzcd}[column sep=2em]
	0 \rar & B \rar & N_n \rar
	& \cdots \rar & N_1 \rar & I \rar & E_{\tau} \rar & R\rar & 0.
	\end{tikzcd}\]
	
	Finally, if $A=B=R$ then the Yoneda product of Yoneda extensions coincides with the cup product of cohomology classes.
\end{proof}

\section{Finite $p$-groups of depth one mod-$p$ cohomology}\label{sec:depthonepgroups}

Let $p$ be an odd prime. For each integer $r$ with $1<r<p-1$, the finite $p$-group $G_r$, described in \eqref{eq:Grassemidirectproduct}, is generated by the elements $\sigma,a_1,\dotsc,a_r $ satisfying the following relations:
\begin{itemize}
\item $\sigma^p=a_i^p=[a_i,a_j]=[a_r,\sigma]=1$, for $i=1,\dotsc,r$ and $j=1,\dotsc,r-1$,
\item $[a_j,\sigma]=a_{j+1}$ for $j=1,\dotsc,r-1$.
\end{itemize}

The aim of this section is to prove Theorem \ref{thm:mainresultintro}. Consider the elementary abelian $p$-group $E=\gen{\sigma,a_r}$ with centralizer $C_{G_r}(E)$ in $G_r$ equal to $E$. By Proposition \ref{prop:DuflotNotbohm}, we have that 
\begin{equation}\label{eq:lowerupperbound}
1=\rk_p(Z(G_r))\leq \depth \Ho(G_r;\Fp)\leq \depth \Ho(C_{G_r}(E);\F_p)=2.
\end{equation}

To show the result, we construct a non-trivial mod-$p$ cohomology class of $G_r$ that restricts trivially to the mod-$p$ cohomology of the centralizers of all rank 2 elementary abelian subgroups of $G_r$. Then, $\omega_d(G_r)=1$ and Theorem \ref{thm:DepthCarlson} yields that $\depth \Ho(G_r;\F_p)=1$.

\subsection{Construction}\label{sec:constructionthetar}

We follow the assumptions in \nameref{notation} and, additionally, suppose that $p>3$. In this section, we construct for each integer $r$ with $1<r<p-1$, a cohomology class $\theta_r \in \operatorname{H}^3(G_r;\F_p)$ that is a cup product of a Yoneda 1-fold extension and a crossed $2$-fold extension.

We start by defining a cohomology class $\sigma^*\in \operatorname{H}^1(G_r;\F_p)=\Hom(G_r,\F_p)$. To that aim, for each $r$, consider the homomorphism $\sigma^*\colon G_r\longrightarrow \Fp$ satisfying
\begin{align*}
\sigma^*(\sigma)=1, \quad \sigma^*(a_1)=\dotsb =\sigma^*(a_r)=0.
\end{align*}
The class $\sigma^*$ can be represented by the Yoneda extension 
\[\begin{tikzcd}
1 \rar & C_p=\gen{a_{r+2}} \rar & C_p\times C_p \rar & C_p=\gen{a_{r+1}} \rar & 1,
\end{tikzcd}\]
where the action of $G_r$ on $C_p\times C_p=\gen{a_{r+1},a_{r+2}}$ is described by 
\begin{displaymath}
\text{for}\; \; g\in G_r, \;\; \text{set}\;\; a_{r+1}^{g}= a_{r+1}a_{r+2}^{\sigma^*(g)}, \quad a_{r+2}^{g}=a_{r+2}.
\end{displaymath}

We continue by defining a crossed 2-fold extension $\eta_r\in \operatorname{H}^2(G_r;\F_p)$ as follows. For $r>1$, let 
\[
\lambda_r\colon T_0/T_{r+1}\times T_0/T_{r+1}\longrightarrow T_0/T_{r+1}
\]
be the alternating bilinear map satisfying $\lambda_r(a_{r-1},a_{r})=a_{r+1}$. Now, define $(T_0/T_{r+1})_{\lambda_r}$ to be the group with underlying set  $T_0/T_{r+1}$ and with group operation given by 
\begin{displaymath}
\text{for} \;\; x,y \in T_0/T_{r+1}\;\; \text{we define,}\;\;  x\cdot_{\lambda_r}y= xy\lambda_r(x,y)^{1/2}.
\end{displaymath}
Finally, define the $p$-group $\widehat G_r=C_p\ltimes (T_0/T_{r+1})_{\lambda_r}$ of size $\abs{\widehat G_r}=p^{r+2}$ and exponent $p$.
Let $\eta_r\in \operatorname{H}^2(G_r,\Fp)$ be the cohomology class represented by the crossed 2-fold extension
\begin{equation}\label{eq:etar}
1 \longrightarrow  C_p=\gen{a_{r+1}} \longrightarrow  \widehat G_r \longrightarrow  G_r \longrightarrow  1.
\end{equation}
Then, we define the cohomology class  $\theta_r=\sigma^*\cup\eta_r\in \operatorname{H}^3(G_r;\Fp)$, which is represented by the crossed 3-fold extension
\begin{equation}\label{eq:thetar}
 1 \longrightarrow  C_p \longrightarrow  C_p\times C_p \longrightarrow  \widehat G_r \longrightarrow  G_r \longrightarrow  1.
\end{equation}

\subsection{Non-triviality}

In the present section we will show that the cohomology class $\theta_r$ described in \eqref{eq:thetar} is non-trivial.

\begin{proposition}\label{prop: theta =/= 0}
	Let $p>3$ be a prime number. For each integer $r$ with $1<r<p-1$, let $\theta_r\in \operatorname{H}^3(G_r;\Fp)$ be the cohomology class constructed in \eqref{eq:thetar}. Then, $\theta_r\not=0$.
\end{proposition}

\begin{proof}
	Assume by contradiction that $\theta_r =0$. Then, by Proposition~\ref{prop: Equivalent Crossed X Diagram} there exists a group $X$ such that the following diagram commutes:
	\[\begin{tikzcd}[column sep={4.em,between origins},row sep=2em]
	& 1 \drar & & & & 1 & \\
	& & C_p\times C_p \drar[swap]{\mu} \arrow[rr] & & \widehat G_r \drar \urar & & \\
	1 \rar & C_p \urar \arrow[dr,Equal] & & X \urar[swap]{\nu} \drar & & G_r \rar & 1 \\
	& & C_p \urar \arrow[rr] & & G_r \arrow[ur,Equal] \drar & & \\
	& 1 \urar &  & & & 1. & 
	\end{tikzcd}\]
	We have that $X=\gen{\bar{\sigma},\bar a_1,\dotsc,\bar a_{r+2}}$ with elements $\bar \sigma,\bar a_1,\dotsc,\bar a_{r+1}, \bar a_{r+2}\in X $ that satisfy
	\[
	\bar a_{r+2}=\mu(a_{r+2}),  \;\; \nu(\bar \sigma)=\sigma \;\; \text{ and }\; \nu(\bar a_i)=a_i \;\text{for all}\;  i=1,\dotsc,r+1,
	\]
	 and we have that $Z(X)=\gen{\bar a_{r+2}}$ and  $\gamma_r(X)=\gen{\bar a_r,\bar a_{r+1}, \bar a_{r+2}}$. Consider the normal subgroup 
	\begin{displaymath}
	Y=\gen{\bar a_{r-1},\bar a_r,\bar a_{r+1}, \bar a_{r+2}}\normaleq X,
	\end{displaymath}
	which fits into the following commutative diagram:
	\[\begin{tikzcd}[column sep={4.em,between origins},row sep=2em]
	& 1 \drar & & & & 1 & \\
	& & \gen{a_{r+1},a_{r+2}} \drar \arrow[rr] & & \gen{a_{r-1},a_r,a_{r+1}} \drar \urar & & \\
	1 \rar & \gen{a_{r+2}} \urar \arrow[dr,Equal] & & Y \urar[swap] \drar & & \gen{a_{r-1},a_r} \rar & 1 \\
	& & \gen{a_{r+2}} \urar \arrow[rr] & & \gen{a_{r-1},a_r} \arrow[ur,Equal] \drar & & \\
	& 1 \urar &  & & & 1. & 
	\end{tikzcd}\]
	Then, we have that $Z(Y)=\gen{\bar a_{r+1},\bar a_{r+2}}$, and moreover,  
	\[
	\; \big[\bar \sigma,Y,\gamma_r(X)\big]=\big[\gamma_r(X),\gamma_r(X)\big]=1 \;\text{ and } \; \big[\gamma_r(X),\bar \sigma,Y\big]=\big[Z(Y),Y\big]=1.
	\]
	Therefore, the three subgroup lemma (see \cite[5.1.10]{Robinson96}) leads us to the conclusion that $\big[Y,\gamma_r(X),\bar \sigma\big]=1$. Nevertheless, a direct computation shows that 
	\begin{displaymath}
	\big[Y,\gamma_r(X),\bar \sigma\big]=\big[Z(Y),\bar\sigma\big]=Z(X)\not=1,
	\end{displaymath}
	which gives a contradiction. Hence, $\theta_r\neq 0$.
\end{proof}

\subsection{Trivial restriction}

In this section we show that for every elementary abelian subgroup $E$ of $G_r$ of $p$-rank $\prk E=2$, the image of $\theta_r$ via the restriction map, 
$$
\res_{C_{G_r}(E)}^{G_r}\colon \operatorname{H}^3(G_r;\F_p) \longrightarrow \operatorname{H}^3(C_{G_r}(E);\F_p),
$$
is trivial, i.e.,  $\res^{G_r}_{C_{G_r}(E)}\theta_r=0$. This will imply that the cohomology class $\theta_r$ is not detected by $\HH_2(G_r)$, and so $\omega_d(G_r)=1$.

\begin{proposition}\label{prop: res C(E) theta = 0}
	Let $p>3$ be a prime number and let $r$ be an integer such that $1<r<p-1$. Let $E\leq G_r$ be an elementary abelian subgroup with $\prk E=2$. Then, $\res^{G_r}_{C_G(E)}\theta_r=0$.
\end{proposition}
\begin{proof}
	There are two types of elementary abelian subgroups $E\leq G_r$, either $E\leq \gen{a_1,\dotsc,a_r}$ or $E\not\leq \gen{a_1,\dotsc,a_r}$. Assume first that $E\leq \gen{a_1,\dotsc,a_r}$. Then, $C_{G_r}(E)=\gen{a_1,\dotsc,a_r}$ and we have that $\res^{G_r}_{C_{G_r}(E)}\sigma^*=0$. Therefore,  
	\[
	\res^{G_r}_{C_{G_r}(E)} \theta_r =(\res^{G_r}_{C_r(E)} \sigma^*)\cup(\res^{G_r}_{C_{G_r}(E)}\eta_r)=0.
	\]
	
	Assume now that $E\not\leq \gen{a_1,\dotsc,a_r}$. Then, $E=\gen{b, a_r}$ with $b=\sigma x$ for some $x\in\gen{a_1,\dotsc,a_{r-1}}$, and $C_{G_r}(E)=E$. Moreover, $\res^{G_r}_{C_G(E)}\eta_r$ is represented by the extension that is obtained by taking the pullback of $\eta_r$ via the inclusion $E\hooklongrightarrow G_r$, as illustrated in the following diagram 
	\[\begin{tikzcd}
	1 \rar & \gen{a_{r+1}} \rar \dar[Equal] & \widehat E =\gen{b, a_r,a_{r+1}} \rar \dar & E \rar \dar & 1 \\
	1 \rar & \gen{a_{r+1}} \rar & \widehat G_r \rar & G_r \rar & 1.
	\end{tikzcd}\]
	Observe that $\widehat E\isom C_p\ltimes(C_p\times C_p)$ is the extraspecial group of order $p^3$ and exponent $p$. Hence, $\res^{G_r}_{C_{G_r}(E)}\eta_r$ is represented by the extension
	\begin{equation}\label{eq:extraspecialextensionclass}
	1 \longrightarrow  C_p=\gen{a_{r+1}} \rightarrow  \widehat E=C_p\ltimes(C_p\times C_p) \rightarrow  C_p\times C_p=\gen{b,a_r} \rightarrow  1.
         \end{equation}
	It can be readily checked (following the construction in \cite[Section IV.3]{Brown82}) that the extension class of \eqref{eq:extraspecialextensionclass} coincides with the cup-product $b^*\cup a_r^*$, and so $\res^{G_r}_{C_{G_r}(E)}\eta_r= b^*\cup a_r^*$. Consequently,
	\begin{displaymath}
	\res^{G_r}_{C_{G_r}(E)}\theta_r = (\res^{G_r}_{C_{G_r}(E)}\sigma^*) \cup b^*\cup a_r^* = 0,
	\end{displaymath}
	as the product of any three elements of degree one is trivial in $\operatorname{H}^3(E;\F_p)$. 
\end{proof}

\begin{proof}[Proof of Theorem \ref{thm:mainresultintro}]
 In \eqref{eq:lowerupperbound}, we obtained that $1\leq \depth\Ho(G_r;\F_p)\leq 2$. In Proposition \ref{prop: theta =/= 0}, we constructed a cohomology class $\theta_r\in \Ho(G_r;\F_p)$ that is non-trivial and that, for every elementary abelian subgroup $E\leq G_r$ of rank 2, satisfies that $\res^{G_r}_{C_{G_r}(E)}\theta_r=0$ (see Proposition \ref{prop: res C(E) theta = 0}). This implies that $ \omega_d(G)=\prk Z(G) = 1$. Then, by Theorem \ref{thm:DepthCarlson}, we conclude that $\depth\Ho(G_r;\Fp)=\prk Z(G)=1$.
\end{proof}

\section{Remarks and further work}

Let $p$ be an odd prime number and let $r$ be an integer with $r\geq p-1$. Consider the finite $p$-groups $G_r=C_p\ltimes T_0/T_r$ defined in \eqref{eq:Grassemidirectproduct}. For each prime $p$, if $r=p-1$, then $G_r$ has size $p^{r+1}$, has exponent $p$ and is of maximal nilpotency class; while if $r>p-1$, then $G$ has size $p^{r+1}$ and exponent bigger than $p$. In particular, for the $p=3$ and $r=2$ case, $G_2$ is the extraspecial $3$-group of order $27$ and exponent $3$, and it is known that the depth of its mod-$3$ cohomology ring is $2$ (compare \cite{Leary92} and \cite{Minh01}). We believe that this phenomena will occur with more generality; namely, for every prime number $p\geq 3$ and $r\geq p-1$, the following equality will hold $\depth \Ho(G_r;\F_p)=2$. For these groups, if we mimic the construction of the mod-$p$ cohomology class $\theta_r$ in Section \ref{sec:constructionthetar}, it is no longer true that its restriction in the mod-$p$ cohomology of the centralizer of all elementary abelian subgroups of $G_r$ of rank 2 vanishes. We propose the following conjecture.

\begin{conjecture}
Let $p$ be an odd prime, let $r\geq p-1$ be an integer, and let 
$$G_r=C_p\ltimes T_0/T_r$$ 
be as in \eqref{eq:Grassemidirectproduct}. Then $\Ho(G_r;\F_p)$ has depth $2$.
\end{conjecture}

The above conjecture is known to be true for the particular cases where $p=3$ and $r=2$ or $r=3$. In these two cases the mod-$p$ cohomology rings have been calculated using computational sources (see \cite{SKing}). Another argument supporting the conjecture is that for a fixed prime $p$ and $r\geq p-1$, the groups $G_r$ have isomorphic mod-$p$ cohomology groups; not as rings, but as $\F_p$-modules (see \cite{Garaialde18}).  This last isomorphism comes from a universal object described in the category of cochain complexes together with a quasi-isomorphism that induces an isomorphism at the level of spectral sequences.

\end{document}